\documentclass{article}

\usepackage{mathptmx}

\usepackage{geometry}
\usepackage{amssymb}
\usepackage{latexsym, amsmath, amscd,amsthm}
\usepackage{graphicx}
\usepackage[percent]{overpic}
\usepackage{units}
\usepackage{xcolor}
\definecolor{linkblue}{HTML}{003d73}
\definecolor{linkgreen}{HTML}{006161}
\definecolor{linkred}{HTML}{a11950}
\usepackage[pagebackref]{hyperref}
\hypersetup{
	pdftitle={Knots with Exactly 10 Sticks},
	pdfauthor={Ryan Blair, Thomas D. Eddy, Nathaniel Morrison, and Clayton Shonkwiler},
	pdfsubject={knot theory},
	pdfkeywords={knots, random polygons, stick number, superbridge index, bridge index},
	colorlinks=true,
	linkcolor=linkblue,
	citecolor=linkgreen,
	urlcolor=linkred
}
\PassOptionsToPackage{caption=false,labelformat=empty}{subfig}
\usepackage[lofdepth]{subfig}
\usepackage[export]{adjustbox}
\usepackage{algorithm}
\usepackage{algorithmicx}
\usepackage{algpseudocode}
\usepackage{booktabs}
\usepackage{authblk}
\usepackage{cite}
\usepackage[nameinlink]{cleveref}

\newtheorem{theorem}{Theorem}

\newtheorem{proposition}[theorem]{Proposition}

\theoremstyle{definition}

\newcommand{\stick}{\operatorname{stick}}

\newcommand{\bridge}{\operatorname{b}}
\newcommand{\superbridge}{\operatorname{sb}}

\title{Knots with Exactly 10 Sticks}
\author[$\ast$]{Ryan Blair}
\author[$\dag$]{Thomas D.\ Eddy}
\author[$\ast$]{Nathaniel Morrison}
\author[$\dag$]{Clayton Shonkwiler}
\affil[$\ast$]{Department of Mathematics, California State University, Long Beach, CA} 
\affil[$\dag$]{Department of Mathematics, Colorado State University, Fort Collins, CO}
\date{}

\begin{document}

\maketitle

\begin{abstract}
	We prove that the knots $13n_{592}$ and $15n_{41,127}$ both have stick number 10. These are the first non-torus prime knots with more than 9 crossings for which the exact stick number is known.
	
\end{abstract}

The \emph{stick number} $\stick(K)$ of a knot $K$ is the minimum number of segments needed to construct a polygonal version of $K$. The stick number was first defined by Randell~\cite{Randell:1994bx} more than two decades ago and is a prototypical \emph{geometric} knot invariant: it shares important qualitative features with other geometric invariants like crossing number, rope length, and minimum M\"obius energy. Since polygonal knots provide a simple model for ring polymers like bacterial DNA (see the survey~\cite{Orlandini:2007kn}), stick number gives an indication of the minimum size of a polymer knot.

As with many elementary invariants, the stick number is quite difficult to compute. It is known for only 35 of the 249 nontrivial knots up to 10 crossings and, prior to our work, the only knots with more than 9 crossings for which it was known are certain torus and composite knots~\cite{Adams:1997gb,Jin:1997da,Bennett:2008th,Adams:2009kp,Johnson:2013ko}. See~\cite{TomClay} for a summary of the key results in this area, as well as a table giving the best known bounds on stick number for all knots up to 10 crossings.

\begin{figure}[htbp]
	\centering
		\subfloat[$13n_{592}$]{\includegraphics[scale=.3,valign=c]{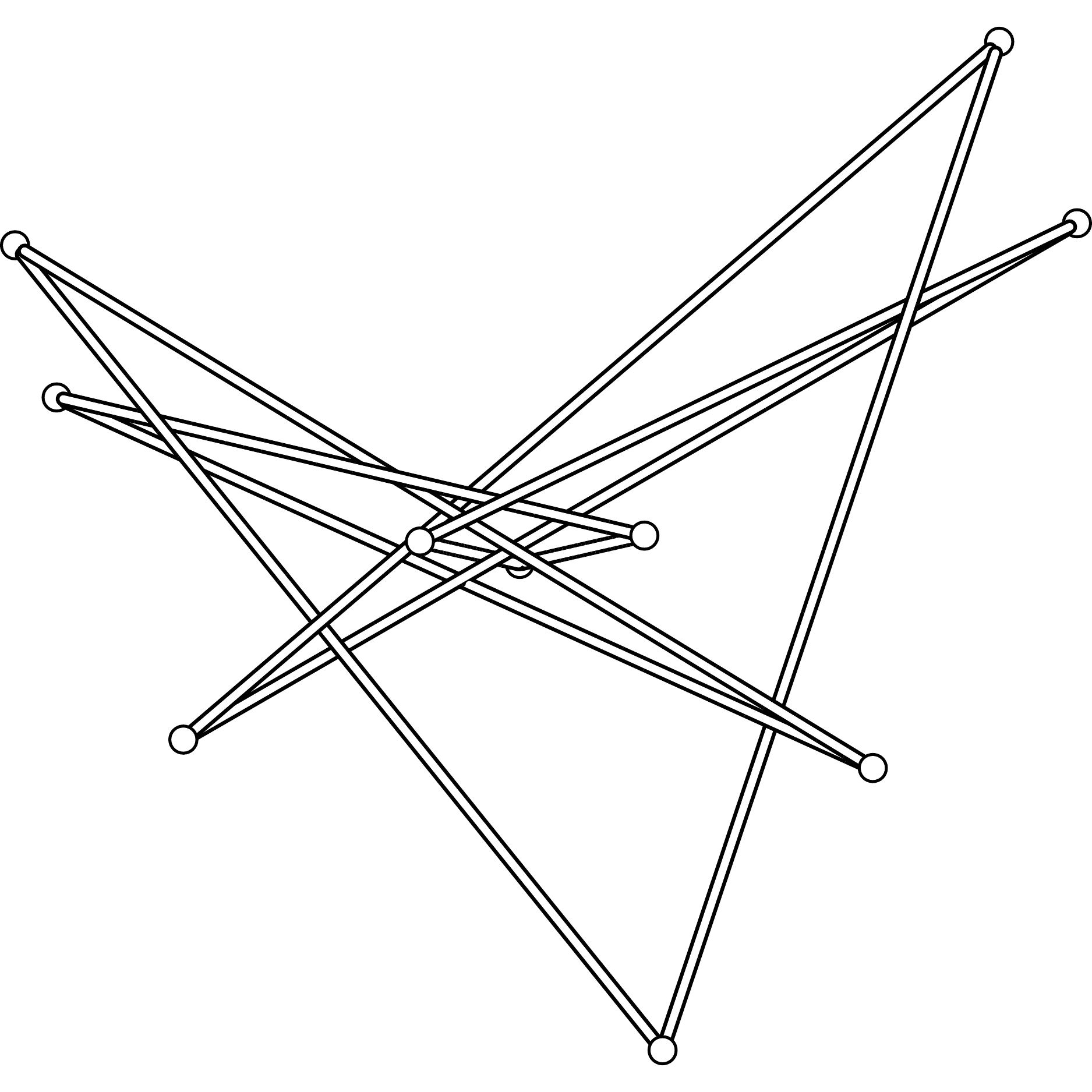}}\qquad
		\subfloat[$15n_{41,127}$]{\includegraphics[scale=.3,valign=c]{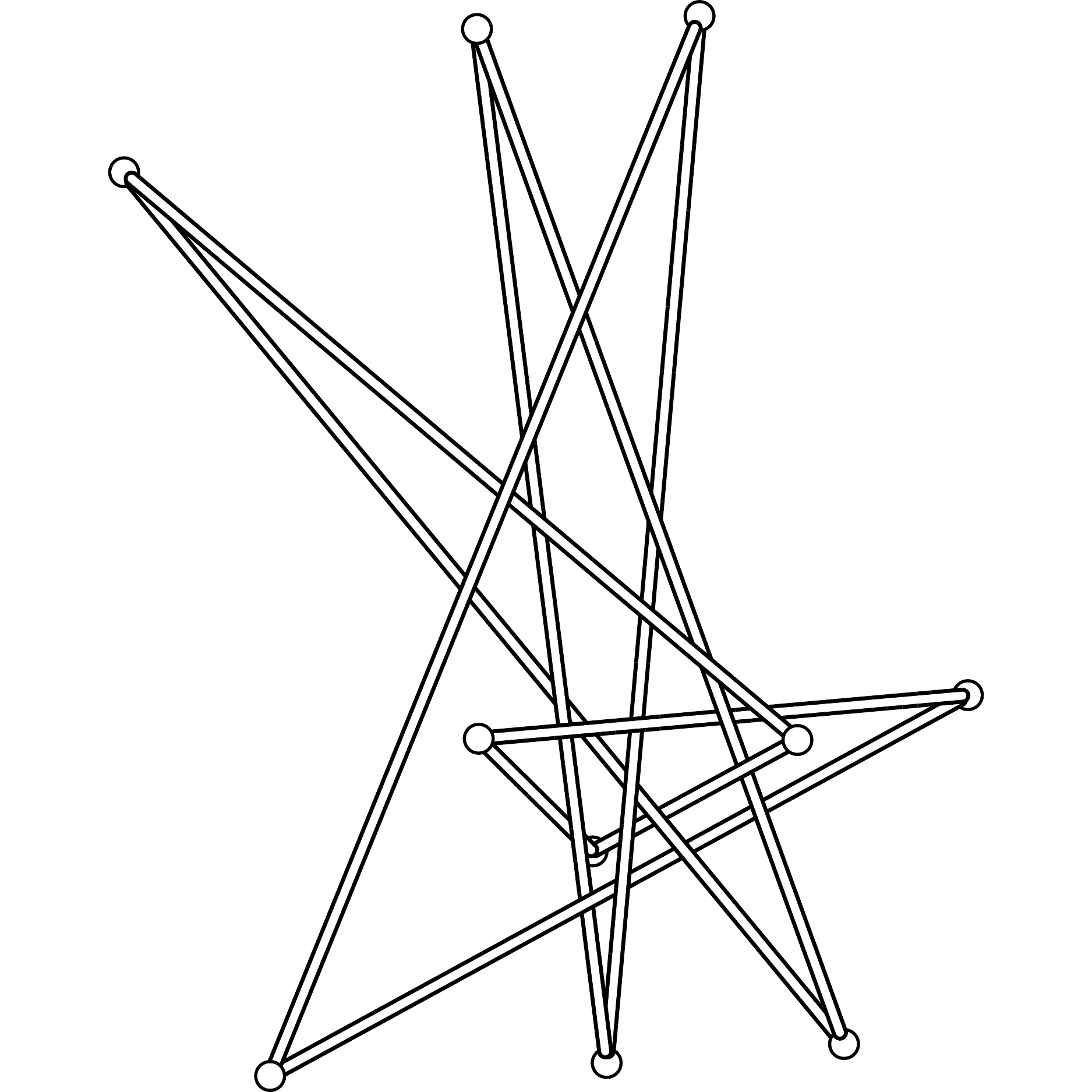}}
	\caption{Equilateral 10-stick examples of $13n_{592}$ and $15n_{41,127}$ shown in orthographic perspective. The $13n_{592}$ is viewed from the direction $(-24,22,3)$ and $15n_{41,127}$ is viewed from the direction $(-10,1,1)$.}
	\label{fig:stick figures}
\end{figure}

The primary goal of this paper is to determine the exact stick numbers of two knots:

\begin{theorem}\label{thm:main}
	The knots $13n_{592}$ and $15n_{41,127}$ have bridge index 4, superbridge index 5, and stick number 10.
\end{theorem}

The basic strategy is to show that 10 is both an upper and a lower bound on stick number for these knots. To see that 10 is an upper bound, it suffices to find 10-stick representatives of each knot, which are pictured in \Cref{fig:stick figures}. Two of us (Eddy and Shonkwiler) discovered these examples while generating very large ensembles of 10-stick random knots in tight confinement for our paper~\cite{TomClay}, in which we searched for new bounds on stick number by generating random polygonal knots confined to small spheres. The point of sampling random knots in confinement is to boost the probability of generating complicated knots. Using symplectic geometry, it turns out that sampling polygonal knots in confinement is equivalent to sampling points in certain convex polytopes according to Lebesgue measure~\cite{Cantarella:2016iy}; see~\cite[Section~3]{TomClay} for a self-contained description of this approach, and see the {\tt stick-knot-gen} project~\cite{stick-knot-gen} for code and supporting data. 

On the other hand, as suggested by the inclusion of statements about these other invariants in \Cref{thm:main}, we will show that 10 is a lower bound on stick number using inequalities relating stick number with bridge index and superbridge index, whose definitions we now recall. Given a knot type $K$, let $\gamma:S^1\rightarrow \mathbb{R}^3$ be a smooth embedding in the ambient isotopy class of $K$. Given a linear function $z:\mathbb{R}^3\rightarrow \mathbb{R}$ of norm one so that $z\circ \gamma$ is Morse, let $b_z(\gamma)$ be the number of local maxima of $z\circ \gamma$. Then the \emph{bridge number} of $\gamma$ is $\bridge(\gamma):=min_{z} b_z(\gamma)$ where the minimum is taken over all $z$ such that $z\circ \gamma$ is Morse. Similarly, the \emph{superbridge number} of $\gamma$ is $\superbridge(\gamma):=max_{z} b_z(\gamma)$. 

We elevate these quantities to knot invariants by minimizing over all $\gamma$ in the ambient isotopy class of $K$. More precisely, the \emph{bridge index} of a knot $K$ is $\bridge(K):=min_{\gamma}\bridge(\gamma)$ and the \emph{superbridge index} of $K$ is ${\superbridge(K):=min_{\gamma}\superbridge(\gamma)}$. The following two propositions relate bridge index, superbridge index, and stick number.

\begin{proposition}[Kuiper~\cite{Kuiper:1987ki}]\label{prop:bridge and superbridge}
	For any nontrivial knot $K$, $\bridge(K) < \superbridge(K)$.
\end{proposition}
\begin{proposition}[Randell~\cite{Randell:1998td}]\label{prop:superbridge and stick}
	For any knot $K$, $\superbridge(K) \leq \frac{1}{2}\stick(K)$.
\end{proposition}

Since $\bridge(K) < \superbridge(K)\leq \frac{1}{2}\stick(K)$ and all three of these invariants are integers, to show a knot has stick number at least $2n$ it suffices to show that the knot has bridge index at least $n - 1$.\footnote{In general this lower bound on stick number is far from sharp: for example, there are infinitely many 2-bridge knots, but only finitely many knots with stick number $\leq n$ for any $n$~\cite{Negami:1991gb,Calvo:2001gv}, so the gap between bridge index and stick number can be arbitrarily large.} Finding lower bounds on bridge index can often be challenging. Classically, Fox colorings provide lower bounds on bridge index; recall that a knot is Fox 3-colorable if and only if there exists a surjective homomorphisms from the knot group to $S_3$, the symmetric group on three elements, that sends meridians of the knot to transpositions. A generalization of this method is to find a surjective homomorphism from the knot group to a group with nice properties. Recently, surjective homomorphisms from knot groups to symmetric groups~\cite{Baader:2019vm} and Coxeter groups~\cite{Baader:2019uf} that send meridians to elements of order two have been used to give lower bounds on bridge index. An example of this approach is the following well-known result.

\begin{proposition}\label{prop:homomorphism}
	Let $K$ be a knot and $S_n$ the symmetric group on $n$ elements. If the knot group $\pi_1\!\left(S^3\, \backslash K\right)\!$ admits a surjective homomorphism to $S_n$ such that every meridian is sent to a transposition, then $\bridge(K) \geq n-1$.
\end{proposition}

In ongoing work~\cite{RyanNateSashka}, two of us (Blair and Morrison) together with Alexandra Kjuchukova are building on earlier algorithms giving upper bounds on bridge index~\cite{Blair:2017ud} to develop code that gives matching lower bounds by searching for all homomorphisms to finite Coxeter groups for knots with 16 or fewer crossings. An incomplete version of this code found the homomorphism from the knot group of $15n_{41,127}$ to the symmetric group $S_5$ given below.

More broadly, the stick knot data~\cite{stick-knot-gen} and the forthcoming Coxeter group homomorphism data are each substantial but mostly unexplored, and we believe they will yield further interesting results, both separately and in combination.

\begin{table}
	\begin{center}
	\begin{tabular}{|rrr|}
		\multicolumn{3}{c}{$13n_{592}$} \\
		\hline
	    0& 0& 0 \\
	     \(10,\!000,\!000\)& 0& 0 \\
	     \(1,\!442,\!849\)&\(5,\!174,\!472\)& 0 \\
	     \(8,\!382,\!194\)&\(-1,\!845,\!927\)&\(1,\!599,\!839\) \\
	     \(2,\!602,\!230\)&\(3,\!136,\!071\)&\(-4,\!863,\!265\) \\
	     \(3,\!168,\!808\)&\(5,\!301,\!213\)&\(4,\!883,\!076\) \\
	     \(3,\!990,\!825\)&\(284,\!477\)&\(-3,\!728,\!350\) \\
	     \(4,\!424,\!114\)&\(7,\!009,\!518\)&\(3,\!659,\!890\) \\
	     \(1,\!311,\!109\)&\(-1,\!250,\!555\)&\(-1,\!039,\!088\) \\
	     \(7,\!812,\!682\)&\(6,\!236,\!842\)&\(252,\!592\) \\
		 \hline
	\end{tabular}
	\qquad \qquad
	\begin{tabular}{|rrr|}
		\multicolumn{3}{c}{$15n_{41,127}$} \\
		\hline
	    0 & 0 & 0 \\
	     \(10,\!000,\!000\) & 0 & 0 \\
	     \(388,\!857\) & \(2,\!761,\!507\) & 0 \\
	     \(5,\!630,\!173\) & \(-3,\!679,\!476\) & \(-5,\!571,\!566\) \\
	     \(5,\!516,\!344\) & \(2,\!658,\!376\) & \(2,\!162,\!688\) \\
	     \(1,\!049,\!311\) & \(-120,\!662\) & \(-6,\!341,\!579\) \\
	     \(6,\!132,\!824\) & \(510,\!618\) & \(2,\!246,\!749\) \\
	     \(1,\!408,\!716\) & \(1,\!799,\!173\) & \(-6,\!472,\!335\) \\
	     \(1,\!932,\!613\) & \(-2,\!026,\!605\) & \(2,\!752,\!032\) \\
	     \(9,\!330,\!316\) & \(3,\!365,\!085\) & \(-1,\!273,\!345\) \\
		 \hline
	\end{tabular}
	\end{center}
	\caption{Coordinates of the vertices of the 10-stick $13n_{592}$ and $15n_{41,127}$ given as integers. To get the vertices of a stick knot with unit-length edges, multiply each coordinate by $10^{-7}$. The coordinates can be downloaded from the {\tt stick-knot-gen} project~\cite{stick-knot-gen} either as tab-separated text files or in a SQLite database.}
	\label{tab:coords}
\end{table}

\begin{proof}[Proof of \Cref{thm:main}]
	We give the coordinates of the vertices of 10-stick representatives of both knots in \Cref{tab:coords} and show pictures in \Cref{fig:stick figures}, proving that $\stick(K) \leq 10$ for both of these knots.\footnote{In fact, the edges of these representatives are all (numerically) the same length and it is straightforward to prove the existence of a true equilateral stick knot of the same type using a result of Millett and Rawdon~\cite[Corollary~2]{Millett:2003kl}, so it will follow that the equilateral stick number of these knots is also equal to 10.} Combining this with Propositions~\ref{prop:bridge and superbridge} and~\ref{prop:superbridge and stick}, we see that
	\[
		\bridge(K) < \superbridge(K) \leq \frac{1}{2} \stick(K) \leq 5,
	\]
	for both knots, so the result follows if we can show that $\bridge(K) \geq 4$. 
	
\begin{figure}[htbp]
	\centering
		\includegraphics[scale=.645]{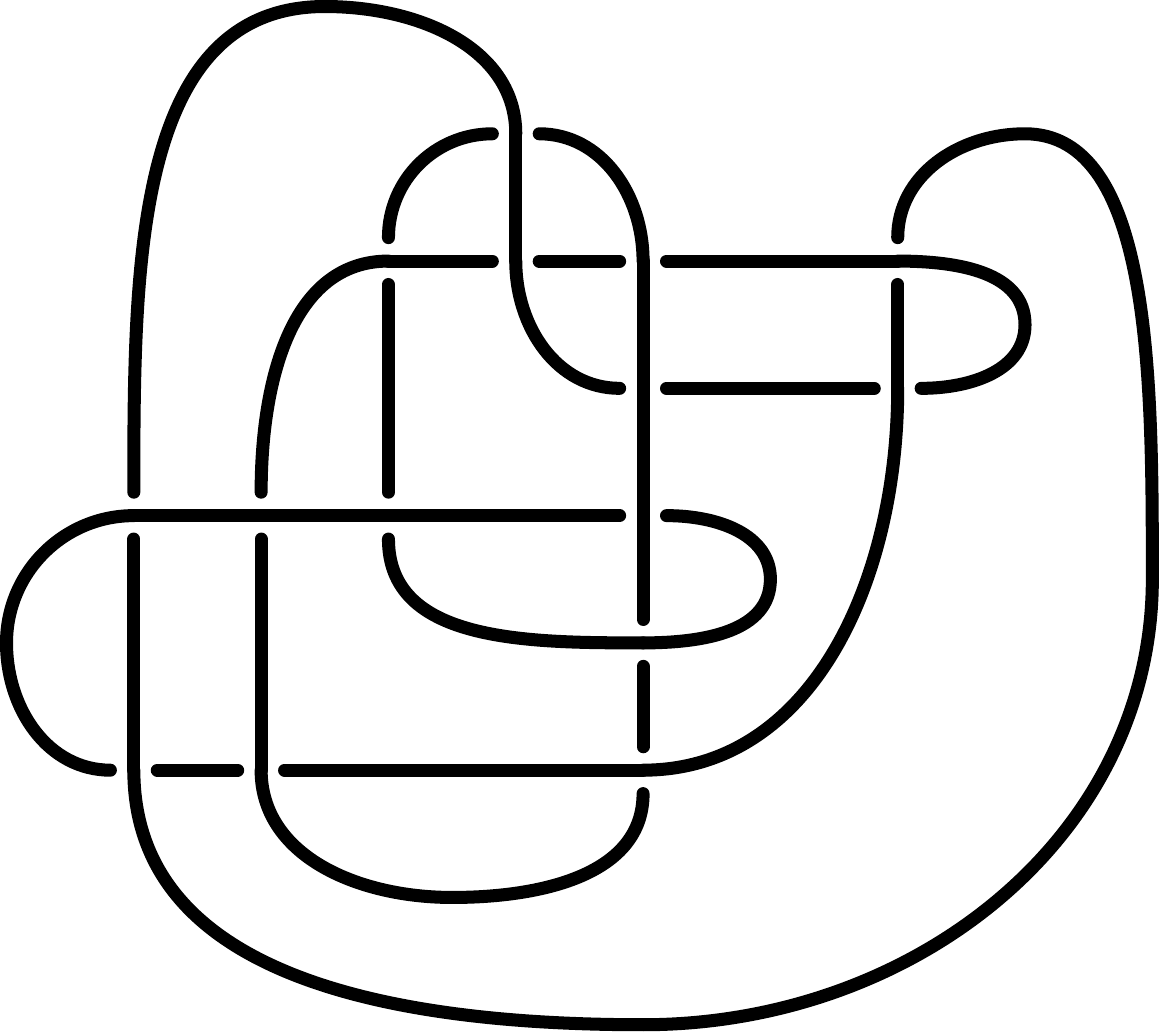}
		\put(-150,99){$1$}
		\put(-103,99){$2$}
		\put(-103,75){$3$}
		\put(-150,146){$4$}
		\put(-128,170){$5$}
		\put(-104,146){$6$}
		\put(-103,122){$7$}
		\put(-103,51){$8$}
		\put(-174,51){$9$}
		\put(-179,99){$10$}
		\put(-131,146){$11$}
		\put(-61,146){$12$}
		\put(-61,122){$13$}
		\put(-203,99){$14$}
		\put(-202,51){$15$}
		\put(-235,75){$\mathbf{(1\,3)}$}
		\put(-202,175){$\mathbf{(2\,5)}$}
		\put(-180,135){$\mathbf{(1\,2)}$}
		\put(-146,66){$\mathbf{(1\,4)}$}
		\put(-57,70){$(1\,2)$}
		\put(-103.5,165){$(3\,4)$}
		\put(-31,144){$(1\,5)$}
		\put(-40,9){$(2\,5)$}
		\put(-188,39){$(1\,3)$}
		\put(-156,165){$(3\,4)$}
		\put(-142.5,119){$(3\,4)$}
		\put(-95,58){$(1\,3)$}
		\put(-146,15){$(2\,3)$}
		\put(-117,134){$(1\,5)$}
		\put(-81,110){$(2\,5)$}
	\caption{The {\tt SnapPy}~\cite{snappy} diagram of $15n_{41,127}$ with crossings labeled $1,2,\dots , 15$ and with strand labels defining a homomorphism $\pi_1\!\left(S^3 \,\backslash 15n_{41,127}\right)\! \twoheadrightarrow S_5$. The generating labels are bolded.}
	\label{fig:k15n41127 Wirtinger}
\end{figure}
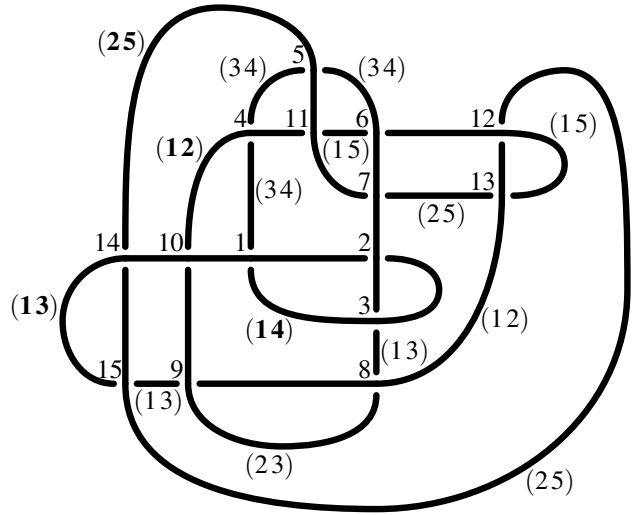
	
	To do so, we use the strategy outlined above and seek a surjective homomorphism from each knot group to the symmetric group $S_5$. \Cref{fig:k15n41127 Wirtinger} shows a diagram of $15n_{41,127}$ with strand labels defining a homomorphism to $S_5$. Specifically, we started by labeling the strand $(-10,4,-11)$ with the transposition $(1\,2)$, the strand $(-15,14,10,1,-2)$ with $(1\,3)$, the strand $(-2,3,-1)$ with $(1\,4)$, and $(-7,11,5,-14)$ with $(2\,5)$, then propogated this to a complete labeling of strands via the Wirtinger relations. This labeling satisfies the Wirtinger relations and the transpositions $\{(1\,2),(1\,3),(1\,4),(2\,5)\}$ generate $S_5$, so this defines a surjective homomorphism ${\pi_1\!\left(S^3 \,\backslash 15n_{41,127}\right)\! \twoheadrightarrow S_5}$. Therefore, \Cref{prop:homomorphism} implies that $\bridge(15n_{41,127})\!\geq 4$.
	
\begin{figure}[htbp]
	\centering
		\subfloat[$15n_{41,127}$]{\includegraphics[scale=.5,valign=c]{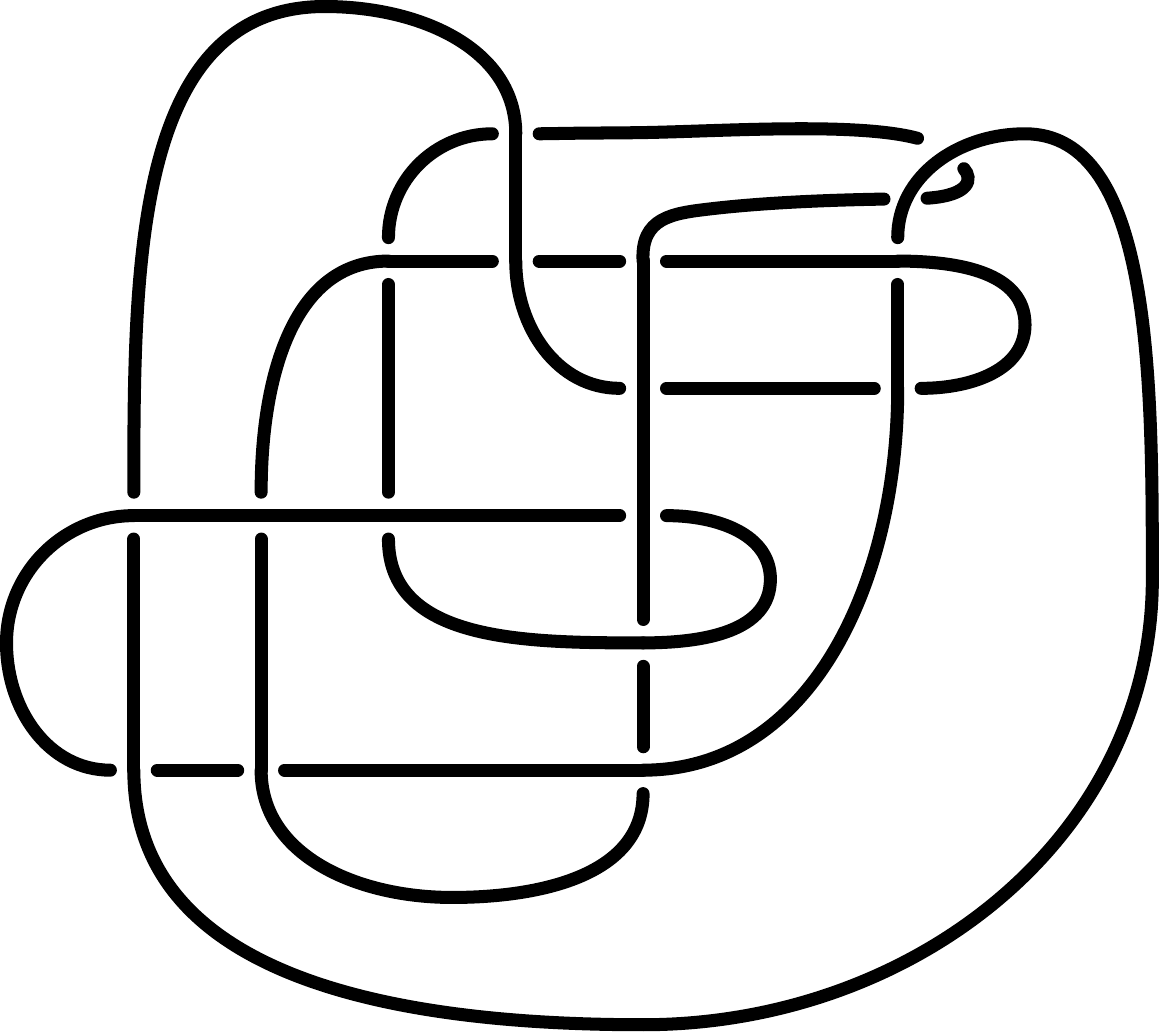}}\qquad  \qquad \qquad
		\subfloat[$\overline{13n_{592}}$]{\includegraphics[scale=.5,valign=c]{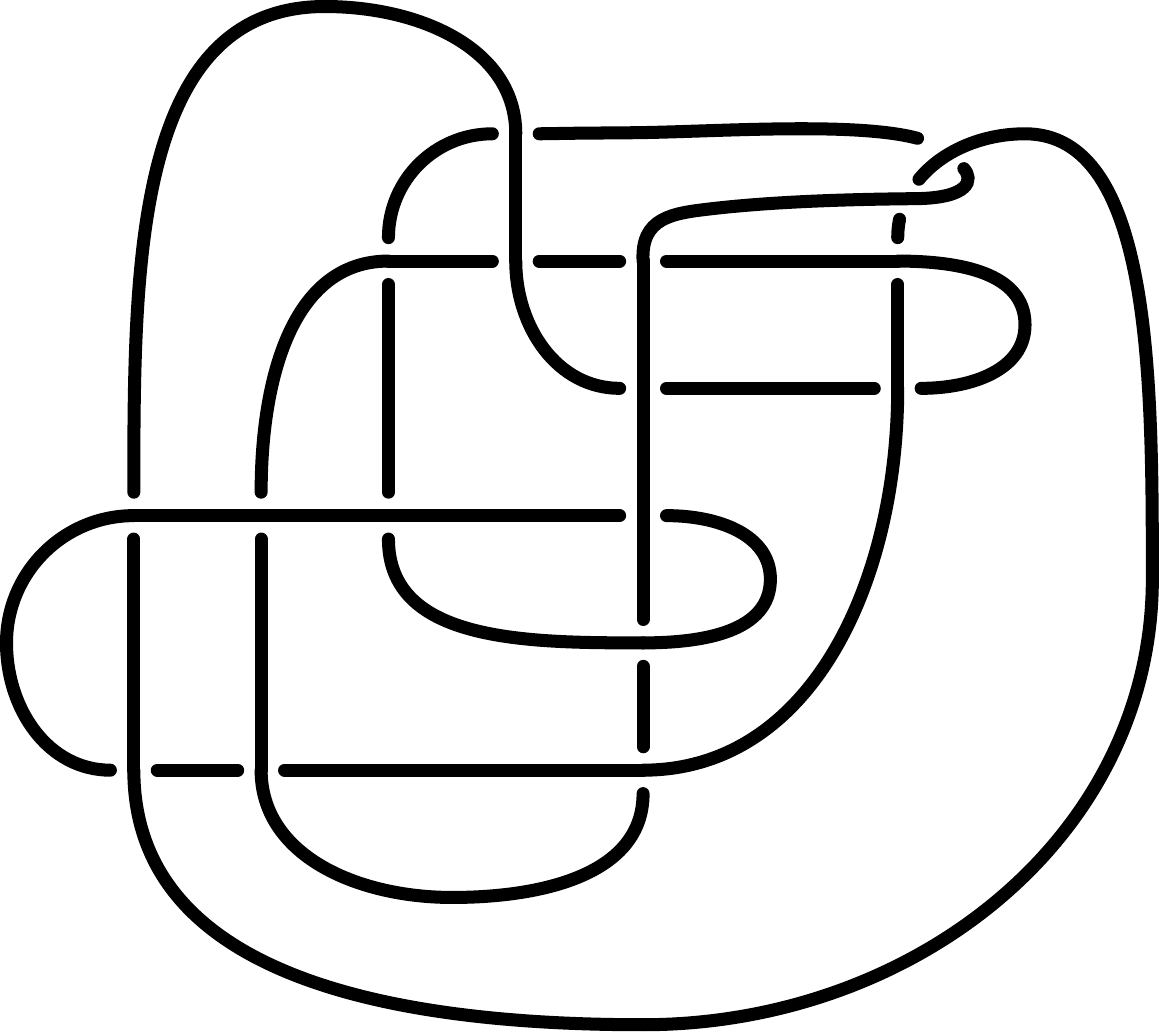}}
	\caption{A non-minimal diagram of $15n_{41,127}$, and a diagram of $\overline{13n_{592}}$ produced by changing a crossing.}
	\label{fig:Reidemeister}
\end{figure}	
	
	\Cref{fig:Reidemeister} shows a non-minimal diagram of $15n_{41,127}$ on the left. This diagram is produced by applying a Reidemeister II move to the diagram in \Cref{fig:k15n41127 Wirtinger}. Changing a single crossing yields a diagram of $\overline{13n_{592}}$, the mirror of $13n_{592}$. Extending the labeling from \Cref{fig:k15n41127 Wirtinger} to this diagram -- which is possible since the transpositions $(3\, 4)$ and $(2\, 5)$ commute -- defines a surjective homomorphism $\pi_1\!\left(S^3 \,\backslash \overline{13n_{592}}\right)\! \twoheadrightarrow S_5$, so $\bridge\!\left(13n_{592}\right)\! = \bridge\!\left(\overline{13n_{592}}\right)\! \geq 4$ as well.
\end{proof}

Somewhat surprisingly, changing $15n_{41,127}$ into $\overline{13n_{592}}$ by switching a single crossing can actually be realized at the level of stick knots as seen in \Cref{fig:mirror}. The non-equilateral 10-stick representation of $\overline{13n_{592}}$ comes from moving the ninth vertex of the $15n_{41,127}$ given in \Cref{tab:coords} to $(3,\!708,\!061, -732,600, 1,\!785,\!942)$ and leaving all other vertices unchanged.

\begin{figure}[t]
	\centering
		\subfloat[$15n_{41,127}$]{\includegraphics[scale=.3,valign=c]{K15n41127rotated.pdf}}\qquad 
		\subfloat[$\overline{13n_{592}}$]{\includegraphics[scale=.3,valign=c]{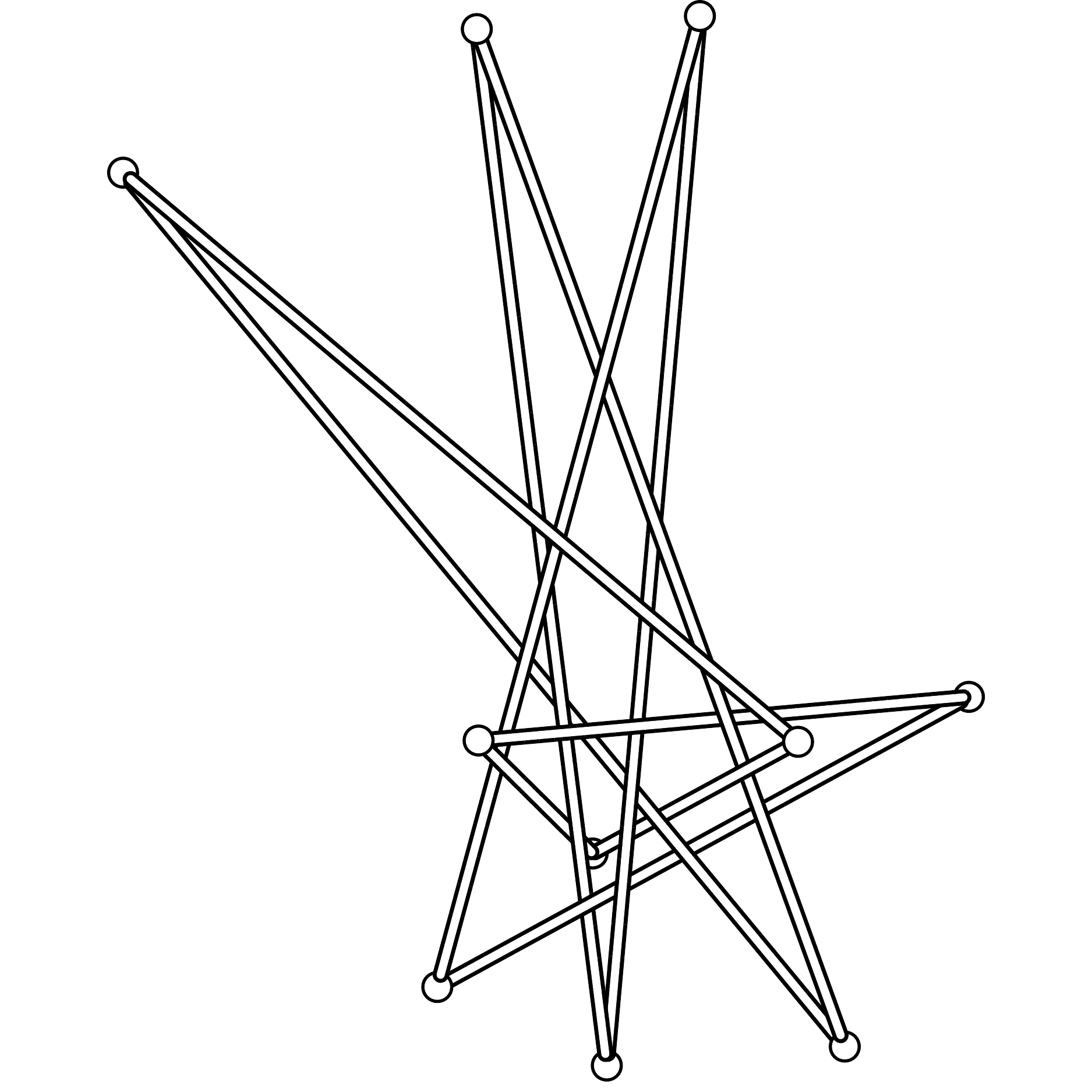}}
	\caption{Moving the ninth vertex of $15n_{41,127}$ produces a non-equilateral 10-stick $\overline{13n_{592}}$.}
	\label{fig:mirror}
\end{figure}

Since some of the labels in \Cref{fig:k15n41127 Wirtinger} don't change at undercrossings, we can switch the signs of any or all of these crossings and produce surjective homomorphisms to $S_5$ from other knot groups. There exist 11-stick representatives of some of these knots, which allows us to compute their bridge index and superbridge index.

\begin{proposition}\label{prop:new superbridge}
	The knots $13n_{285}$, $13n_{293}$, $13n_{587}$, $13n_{607}$, $13n_{611}$, $13n_{835}$, $13n_{1177}$, $13n_{1192}$, and $15n_{41,126}$ all have bridge index 4 and superbridge index 5. In each case the stick number is either 10 or 11.
\end{proposition}

\begin{proof}
	We give coordinates of the vertices of equilateral 11-stick realizations of each of these knots in \Cref{sec:11-stick coords}, so we have $\bridge(K) < \superbridge(K) \leq \frac{1}{2} \stick(K) \leq \frac{11}{2}$ and the result follows if we can show $\bridge(K) \geq 4$. 
	
	To see this, observe that we can change the sign of crossing 4 in \Cref{fig:k15n41127 Wirtinger} to get a diagram for $13n_{835}$ and associated surjective homomorphism $\pi_1\!\left(S^3 \, \backslash 13n_{835}\right)\! \twoheadrightarrow S_5$, so $\bridge\!\left(13n_{835}\right)\! \geq 4$ by \Cref{prop:homomorphism}. As summarized in \Cref{tab:crossing changes}, other crossing changes produce diagrams and homomorphisms for the rest of the knots in the \nameCref{prop:new superbridge} (or their mirrors) and the result follows.
\end{proof}

\begin{table}[h]
	\begin{center}
		\begin{tabular}{l@{\quad}l}
			Knot & Change crossing(s)\dots \\
			\midrule
			$13n_{285}$ & 5, 14, and 15 \\
			$13n_{293}$ & 4, 6, 7, 14, and 15 \\
			$13n_{587}$ & 4, 5, and 6 \\
			$13n_{607}$ & 4, 5, and 7 \\
			$13n_{611}$ & 4, 14, and 15 \\
			$13n_{835}$ & 4 \\
			$\overline{13n_{1177}}$ & 14 \\
			$\overline{13n_{1192}}$ & 5, 7, and 14 \\
			$15n_{41,126}$ & 14 and 15
		\end{tabular}
	\end{center}
	\caption{Crossing changes in \Cref{fig:k15n41127 Wirtinger} produce several other knots, each of which must have bridge index at least 4.}
	\label{tab:crossing changes}
\end{table}

\subsection*{Acknowledgments} 
\label{sub:acknowledgments}

We are grateful to Jason Cantarella and Alexandra Kjuchukova for ongoing conversations and we thank the KnotInfo project~\cite{knotinfo} for catalyzing the present paper. This work was partially supported by grants from the National Science Foundation (DMS--1821254, RB and NM) and the Simons Foundation (\#354225, CS).

\appendix


\section{Coordinates of the vertices of 11-stick knots} 
\label{sec:11-stick coords}

We give coordinates of equilateral 11-stick representations of each of the knots mentioned in \Cref{prop:new superbridge}. These coordinates can be downloaded from the {\tt stick-knot-gen} project~\cite{stick-knot-gen} either as tab-separated text files or in a SQLite database.

\small
\begin{center}
	\begin{tabular}{|rrr|}
		\multicolumn{3}{c}{$13n_{285}$} \\
		\hline
		$0$ & $0$ & $0$ \\
		$10,\!000,\!000$ & $0$ & $0$ \\
		$859,\!828$ & $4,\!056,\!755$ & $0$ \\
		$3,\!341,\!572$ & $-4,\!982,\!897$ & $3,\!482,\!189$ \\
		$6,\!315,\!660$ & $167,\!473$ & $-4,\!556,\!996$ \\
		$1,\!071,\!184$ & $2,\!581,\!206$ & $3,\!608,\!135$ \\
		$3,\!078,\!763$ & $-4,\!845,\!959$ & $-2,\!779,\!888$ \\
		$5,\!080,\!293$ & $3,\!832,\!695$ & $1,\!767,\!071$ \\
		$-1,\!760,\!951$ & $-3,\!454,\!009$ & $2,\!085,\!376$ \\
		$5,\!721,\!544$ & $2,\!372,\!114$ & $-1,\!087,\!723$ \\
		$6,\!595,\!292$ & $-4,\!997,\!566$ & $5,\!614,\!843$ \\
		\hline
	\end{tabular}
	\qquad\qquad
	\begin{tabular}{|rrr|}
		\multicolumn{3}{c}{$13n_{293}$} \\
		\hline
		$0$ & $0$ & $0$ \\
		$10,\!000,\!000$ & $0$ & $0$ \\
		$702,\!140$ & $3,\!681,\!005$ & $0$ \\
		$5,\!902,\!391$ & $-4,\!422,\!027$ & $2,\!701,\!530$ \\
		$4,\!828,\!936$ & $5,\!481,\!886$ & $1,\!829,\!635$ \\
		$2,\!515,\!302$ & $-3,\!074,\!450$ & $-2,\!800,\!289$ \\
		$6,\!012,\!626$ & $-299,\!876$ & $6,\!147,\!920$ \\
		$7,\!106,\!044$ & $-248,\!388$ & $-3,\!791,\!988$ \\
		$666,\!709$ & $2,\!731,\!194$ & $3,\!254,\!788$ \\
		$5,\!345,\!394$ & $-2,\!852,\!898$ & $-3,\!595,\!600$ \\
		$8,\!348,\!109$ & $5,\!360,\!542$ & $1,\!254,\!462$ \\
		\hline
	\end{tabular}
	
	\begin{tabular}{|rrr|}
		\multicolumn{3}{c}{$13n_{587}$} \\
		\hline
		$0$ & $0$ & $0$ \\
		$10,\!000,\!000$ & $0$ & $0$ \\
		$1,\!117,\!558$ & $4,\!593,\!716$ & $0$ \\
		$8,\!904,\!479$ & $-1,\!236,\!400$ & $2,\!318,\!106$ \\
		$5,\!218,\!343$ & $3,\!791,\!336$ & $-5,\!500,\!735$ \\
		$5,\!234,\!842$ & $3,\!507,\!047$ & $4,\!495,\!209$ \\
		$7,\!948,\!983$ & $-4,\!586,\!685$ & $-712,\!951$ \\
		$5,\!383,\!084$ & $5,\!067,\!290$ & $-1,\!178,\!693$ \\
		$4,\!212,\!386$ & $-3,\!577,\!785$ & $3,\!709,\!265$ \\
		$2,\!513,\!093$ & $100,\!976$ & $-5,\!432,\!897$ \\
		$8,\!316,\!823$ & $5,\!514,\!226$ & $650,\!970$ \\
		\hline
	\end{tabular}
	\qquad\qquad
	\begin{tabular}{|rrr|}
		\multicolumn{3}{c}{$13n_{607}$} \\
		\hline
		$0$ & $0$ & $0$ \\
		$10,\!000,\!000$ & $0$ & $0$ \\
		$986,\!011$ & $4,\!329,\!896$ & $0$ \\
		$8,\!208,\!034$ & $-2,\!392,\!018$ & $1,\!630,\!418$ \\
		$2,\!621,\!993$ & $5,\!354,\!852$ & $-1,\!333,\!051$ \\
		$-1,\!613,\!778$ & $841,\!929$ & $6,\!521,\!360$ \\
		$-781,\!841$ & $2,\!012,\!324$ & $-3,\!375,\!006$ \\
		$6,\!224,\!891$ & $1,\!973,\!665$ & $3,\!759,\!713$ \\
		$-3,\!034,\!534$ & $5,\!688,\!727$ & $3,\!080,\!474$ \\
		$3,\!594,\!430$ & $-618,\!875$ & $-953,\!258$ \\
		$3,\!590,\!235$ & $9,\!333,\!254$ & $24,\!042$ \\
		\hline
	\end{tabular}

	\begin{tabular}{|rrr|}
		\multicolumn{3}{c}{$13n_{611}$} \\
		\hline
		$0$ & $0$ & $0$ \\
		$10,\!000,\!000$ & $0$ & $0$ \\
		$1,\!499,\!082$ & $5,\!266,\!345$ & $0$ \\
		$2,\!621,\!286$ & $-4,\!670,\!135$ & $83,\!746$ \\
		$9,\!409,\!241$ & $2,\!634,\!889$ & $-664,\!786$ \\
		$187,\!148$ & $-871,\!197$ & $966,\!280$ \\
		$8,\!269,\!950$ & $4,\!801,\!218$ & $-612,\!336$ \\
		$2,\!642,\!458$ & $-3,\!367,\!232$ & $655,\!639$ \\
		$4,\!915,\!758$ & $6,\!263,\!837$ & $2,\!095,\!993$ \\
		$806,\!026$ & $-2,\!079,\!111$ & $-1,\!578,\!968$ \\
		$8,\!522,\!161$ & $-2,\!123,\!012$ & $4,\!781,\!798$ \\
		\hline
	\end{tabular}
	\qquad\qquad
	\begin{tabular}{|rrr|}
		\multicolumn{3}{c}{$13n_{835}$} \\
		\hline
		$0$ & $0$ & $0$ \\
		$10,\!000,\!000$ & $0$ & $0$ \\
		$897,\!936$ & $4,\!141,\!549$ & $0$ \\
		$6,\!400,\!763$ & $-2,\!074,\!978$ & $5,\!574,\!377$ \\
		$6,\!030,\!097$ & $512,\!087$ & $-4,\!078,\!068$ \\
		$257,\!946$ & $5,\!028,\!434$ & $2,\!725,\!231$ \\
		$501,\!111$ & $-3,\!773,\!276$ & $-2,\!015,\!104$ \\
		$3,\!966,\!743$ & $5,\!006,\!086$ & $1,\!288,\!257$ \\
		$2,\!987,\!114$ & $-4,\!685,\!303$ & $3,\!550,\!403$ \\
		$1,\!478,\!845$ & $3,\!843,\!801$ & $-1,\!447,\!548$ \\
		$8,\!865,\!166$ & $785,\!407$ & $4,\!559,\!821$ \\
		\hline
	\end{tabular}

	\begin{tabular}{|rrr|}
		\multicolumn{3}{c}{$13n_{1177}$} \\
		\hline
		$0$ & $0$ & $0$ \\
		$10,\!000,\!000$ & $0$ & $0$ \\
		$4,\!452,\!828$ & $8,\!320,\!390$ & $0$ \\
		$-263,\!502$ & $-219,\!251$ & $2,\!197,\!901$ \\
		$6,\!301,\!576$ & $2,\!832,\!197$ & $-4,\!700,\!535$ \\
		$685,\!003$ & $5,\!395,\!073$ & $3,\!166,\!216$ \\
		$5,\!870,\!985$ & $-2,\!081,\!859$ & $-981,\!201$ \\
		$4,\!018,\!052$ & $7,\!077,\!471$ & $-4,\!541,\!159$ \\
		$3,\!945,\!828$ & $-441,\!769$ & $2,\!050,\!944$ \\
		$2,\!482,\!141$ & $5,\!602,\!190$ & $-5,\!780,\!288$ \\
		$8,\!891,\!711$ & $4,\!219,\!576$ & $1,\!769,\!930$ \\
		\hline
	\end{tabular}
	\qquad\qquad
	\begin{tabular}{|rrr|}
		\multicolumn{3}{c}{$13n_{1192}$} \\
		\hline
		$0$ & $0$ & $0$ \\
		$10,\!000,\!000$ & $0$ & $0$ \\
		$2,\!894,\!741$ & $7,\!036,\!710$ & $0$ \\
		$5,\!839,\!468$ & $-999,\!634$ & $-5,\!171,\!631$ \\
		$3,\!390,\!487$ & $-2,\!857,\!690$ & $4,\!344,\!152$ \\
		$6,\!766,\!649$ & $3,\!572,\!004$ & $-2,\!530,\!479$ \\
		$3,\!148,\!686$ & $-5,\!725,\!554$ & $-1,\!848,\!025$ \\
		$7,\!784,\!297$ & $2,\!619,\!976$ & $1,\!129,\!091$ \\
		$2,\!939,\!835$ & $1,\!169,\!775$ & $-7,\!498,\!081$ \\
		$7,\!317,\!510$ & $-555,\!513$ & $1,\!325,\!713$ \\
		$4,\!748,\!190$ & $8,\!652,\!618$ & $-1,\!608,\!380$ \\
		\hline
	\end{tabular}

	\begin{tabular}{|rrr|}
		\multicolumn{3}{c}{$15n_{41,126}$} \\
		\hline
		$0$ & $0$ & $0$ \\
		$10,\!000,\!000$ & $0$ & $0$ \\
		$2,\!780,\!458$ & $6,\!919,\!409$ & $0$ \\
		$7,\!526,\!255$ & $-1,\!676,\!113$ & $1,\!895,\!896$ \\
		$-912,\!130$ & $3,\!222,\!643$ & $-294,\!044$ \\
		$8,\!983,\!359$ & $1,\!841,\!657$ & $-708,\!974$ \\
		$956,\!597$ & $-660,\!225$ & $4,\!705,\!052$ \\
		$5,\!336,\!611$ & $295,\!308$ & $-4,\!233,\!764$ \\
		$6,\!738,\!903$ & $3,\!591,\!613$ & $5,\!102,\!613$ \\
		$4,\!279,\!082$ & $-2,\!033,\!573$ & $-2,\!790,\!837$ \\
		$8,\!155,\!809$ & $4,\!032,\!625$ & $4,\!149,\!784$ \\
		\hline
	\end{tabular}
\end{center}

\normalsize

\bibliography{stickknots-special,stickknots}

\end{document}